\documentclass[12pt,leqno]{article}
\usepackage{amsfonts}
\usepackage{amsmath, dsfont, amsthm, amsfonts, amssymb, color}
\usepackage{mathrsfs}
\usepackage{color}
\usepackage{stmaryrd}
\newtheorem{thm}{Theorem}[section]

\newtheorem{lem}[thm]{Lemma}

\newtheorem{rem}[thm]{Remark}
\theoremstyle{definition}

\newcommand{\scr}[1]{\mathscr #1}
\definecolor{wco}{rgb}{0.5,0.2,0.3}

\numberwithin{equation}{section} \theoremstyle{remark}

\newcommand{\ua}{\uparrow}

\title{{\bf
  	Entropy-cost type Propagation of Chaos for Mean Field Particle Systems with Bounded Measurable Interaction}\footnote{Supported in
 part by  National Key R\&D Program of China (No. 2022YFA1006000) and NNSFC (12271398).} }
\author{
{\bf   Xing Huang  }\\
\footnotesize{ Center for Applied Mathematics, Tianjin
University, Tianjin 300072, China}\\
\footnotesize{  xinghuang@tju.edu.cn}}
\begin{document}
\allowdisplaybreaks
\def\R{\mathbb R}  \def\ff{\frac} \def\ss{\sqrt} \def\B{\mathbf
B} \def\W{\mathbb W}
\def\N{\mathbb N} \def\kk{\kappa} \def\m{{\bf m}}
\def\ee{\varepsilon}\def\ddd{D^*}
\def\dd{\delta} \def\DD{\Delta} \def\vv{\varepsilon} \def\rr{\rho}
\def\<{\langle} \def\>{\rangle} \def\GG{\Gamma} \def\gg{\gamma}
  \def\nn{\nabla} \def\pp{\partial} \def\E{\mathbb E}
\def\d{\text{\rm{d}}} \def\bb{\beta} \def\aa{\alpha} \def\D{\scr D}
  \def\si{\sigma} \def\ess{\text{\rm{ess}}}
\def\beg{\begin} \def\beq{\begin{equation}}  \def\F{\scr F}
\def\Ric{\text{\rm{Ric}}} \def\Hess{\text{\rm{Hess}}}
\def\e{\text{\rm{e}}} \def\ua{\underline a} \def\OO{\Omega}  \def\oo{\omega}
 \def\tt{\tilde} \def\Ric{\text{\rm{Ric}}}
\def\cut{\text{\rm{cut}}} \def\P{\mathbb P} \def\ifn{I_n(f^{\bigotimes n})}
\def\C{\scr C}      \def\aaa{\mathbf{r}}     \def\r{r}
\def\gap{\text{\rm{gap}}} \def\prr{\pi_{{\bf m},\varrho}}  \def\r{\mathbf r}
\def\Z{\mathbb Z} \def\vrr{\varrho}
\def\L{\scr L}\def\Tt{\tt} \def\TT{\tt}\def\II{\mathbb I}
\def\i{{\rm in}}\def\Sect{{\rm Sect}}  \def\H{\mathbb H}
\def\M{\scr M}\def\Q{\mathbb Q} \def\texto{\text{o}}
\def\Rank{{\rm Rank}} \def\B{\scr B} \def\i{{\rm i}} \def\HR{\hat{\R}^d}
\def\to{\rightarrow}\def\l{\ell}\def\iint{\int}
\def\EE{\scr E}\def\Cut{{\rm Cut}}
\def\A{\scr A} \def\Lip{{\rm Lip}}
\def\BB{\scr B}\def\Ent{{\rm Ent}}\def\L{\scr L}
\def\R{\mathbb R}  \def\ff{\frac} \def\ss{\sqrt} \def\B{\mathbf
B}
\def\N{\mathbb N} \def\kk{\kappa} \def\m{{\bf m}}
\def\dd{\delta} \def\DD{\Delta} \def\vv{\varepsilon} \def\rr{\rho}
\def\<{\langle} \def\>{\rangle} \def\GG{\Gamma} \def\gg{\gamma}
  \def\nn{\nabla} \def\pp{\partial} \def\E{\mathbb E}
\def\d{\text{\rm{d}}} \def\bb{\beta} \def\aa{\alpha} \def\D{\scr D}
  \def\si{\sigma} \def\ess{\text{\rm{ess}}}
\def\beg{\begin} \def\beq{\begin{equation}}  \def\F{\scr F}
\def\Ric{\text{\rm{Ric}}} \def\Hess{\text{\rm{Hess}}}
\def\e{\text{\rm{e}}} \def\ua{\underline a} \def\OO{\Omega}  \def\oo{\omega}
 \def\tt{\tilde} \def\Ric{\text{\rm{Ric}}}
\def\cut{\text{\rm{cut}}} \def\P{\mathbb P} \def\ifn{I_n(f^{\bigotimes n})}
\def\C{\scr C}      \def\aaa{\mathbf{r}}     \def\r{r}
\def\gap{\text{\rm{gap}}} \def\prr{\pi_{{\bf m},\varrho}}  \def\r{\mathbf r}
\def\Z{\mathbb Z} \def\vrr{\varrho}
\def\L{\scr L}\def\Tt{\tt} \def\TT{\tt}\def\II{\mathbb I}
\def\i{{\rm in}}\def\Sect{{\rm Sect}}  \def\H{\mathbb H}
\def\M{\scr M}\def\Q{\mathbb Q} \def\texto{\text{o}} \def\LL{\Lambda}
\def\Rank{{\rm Rank}} \def\B{\scr B} \def\i{{\rm i}} \def\HR{\hat{\R}^d}
\def\to{\rightarrow}\def\l{\ell}
\def\8{\infty}\def\I{1}\def\U{\scr U} \def\n{{\mathbf n}}
\maketitle

\begin{abstract} In this paper, the quantitative entropy-cost type propagation of chaos  for mean field interacting particle system is obtained, where the interaction is assumed to be bounded. More precisely, the relative entropy between the distributions of particle system and the corresponding McKean-Vlasov SDEs at any positive time $t$ depends on $L^1$-transportation cost $\W_1^\Psi$ between the initial distributions. The results weaken the initial assumptions in existing entropy-entropy type propagation of chaos.
 \end{abstract}

\noindent
 AMS subject Classification:\  60H10, 60G44.   \\
\noindent
 Keywords: Mean field interacting particle system, McKean-Vlasov SDEs, quantitative propagation of chaos, relative entropy, bounded interaction.
 \vskip 2cm

\section{Introduction}
The propagation of chaos for mean field interacting particle system characterizes a dynamical mechanism on the time variable of Kac's chaotic property. It is also named Boltzmann's property in \cite{Kac}, where a simplified  homogeneous Boltzmann equation is derived by the approximation of Poisson-like processes originated from dilute monatomic gases. The quantitative convergence rate of the joint distribution of $k$ interacting particles converging to that of the solution to $k$ independent  McKean-Vlasov SDEs attracts much attention. There are various notions of propagation of chaos(strong sense, Wasserstein distance, relative entropy, Fisher information, etc.), see for instance \cite{HS,JW2} for more details.

For any Polish space $(E,\rho)$, let $\scr P(E)$ be the collection of all probability measures on $E$ equipped with the weak topology.
Fix $T>0$. Let $W_t$ be an $n$-dimensional Brownian motion on some complete filtration probability space $(\Omega, \scr F, (\scr F_t)_{t\geq 0},\P)$. $b:[0,T]\times \R^d\times\scr P(\R^d)\to\R^d$, $\sigma:[0,T]\times \R^d\times\scr P(\R^d)\to\R^d\otimes\R^{n}$ are measurable and are bounded on bounded set.

Let $X_0$ be an $\F_0$-measurable $\R^d$-valued random variable,
$N\ge1$ be an integer and $(X_0^i,W^i_t)_{1\le i\le N}$ be i.i.d.\,copies of $(X_0,W_t).$ Consider the mean field interacting particle system
\begin{align}\label{GPS00}\d X^{i,N}_t=b_t(X_t^{i,N}, \hat\mu_t^N)\d t+\sigma_t(X^{i,N}_t, \hat\mu_t^N) \d W^i_t,\ \ 1\leq i\leq N,
\end{align}
where the distribution of $(X_0^{1,N},X_0^{2,N},\cdots,X_0^{N,N} )$ is exchangeable and $\hat\mu_t^N$ is the empirical distribution of $(X_t^{i,N})_{1\leq i\leq N}$, i.e.
\begin{equation*}
 \hat\mu_t^N =\ff{1}{N}\sum_{j=1}^N\dd_{X_t^{j,N}}.
 \end{equation*}
Under some reasonable assumptions, the limit equation of a single particle in \eqref{GPS00} becomes
\begin{align*}\d X_t^i= b_t(X_t^i, \L_{X_t^i})\d t+  \sigma_t(X^i_t,\L_{X_t^i}) \d W^i_t,
\end{align*}
where $\L_{X_t^i}$ is the distribution of $X_t^i$.
Then $X_t^1$ solves the McKean-Vlasov SDE
\begin{align*}\d  X_t=b_t(X_t,\L_{X_t})\mathrm{d} t+\sigma_t(X_t,\L_{X_t})\mathrm{d} W_t.
\end{align*}
This type of SDE, also called distribution dependent SDE, was first introduced in \cite{McKean}. It is now widely applied in neural networks, mean field games and financial mathematics and has been extensively investigated in recent years. One can refer to \cite{LL} and the monograph \cite{WR2024} for various contributions on it.

To quantify the propagation of chaos, one may equip some probability distances on certain subspace of $\scr P(\R^d)$. To this end, for $p>0$, let
$$\scr P_p(\R^d):=\big\{\mu\in \scr P(\R^d): \mu(|\cdot|^p)<\infty\big\},$$
which is a Polish space under the $L^p$-Wasserstein distance
$$\W_p(\mu,\nu)= \inf_{\pi\in \C(\mu,\nu)} \bigg(\int_{\R^{d}\times\R^{d}} |x-y|^p \pi(\d x,\d y)\bigg)^{\ff 1 {p\vee 1}},\ \  \mu,\nu\in \scr P_p(\R^d),$$ where $\C(\mu,\nu)$ is the set of all couplings of $\mu$ and $\nu$.
Let $\alpha\in(0,1]$ and $\rho_\alpha(x,y):=|x-y|^\alpha,\ \ x,y\in\R^d$. Then $(\R^d,\rho_\alpha)$ is a Polish space so that we have the Kantorovich-Rubinstein duality formula
$$\W_\alpha(\gamma,\tilde{\gamma})=\sup_{[f]_\alpha\leq 1}|\gamma(f)-\tilde{\gamma}(f)|,\ \ \gamma,\tilde{\gamma}\in \scr P_\alpha(\R^d),$$
where $[f]_\alpha:=\sup_{x\neq y}\frac{|f(x)-f(y)|}{|x-y|^\alpha}$, see for instance \cite[Theorem 5.10]{Chen} or \cite[Theorem 1.14]{Villani}. We will also use the total variation distance, that is
$$\|\gamma-\tilde{\gamma}\|_{var}=\sup_{|f|\leq 1}|\gamma(f)-\tilde{\gamma}(f)|,\ \ \gamma,\tilde{\gamma}\in \scr P(\R^d).$$
When we consider the interacting particle system, the state space is $(\R^d)^N$ so that some appropriate distance on $\scr P((\R^d)^N)$ is also need to be provided.
For $\alpha\in(0,1]$, define  $$\rho_{\ell_\alpha}:=\sum_{i=1}^k|x^i-y^i|^\alpha,\ \ x=(x^1,x^2,\cdots,x^k), y=(y^1,y^2,\cdots,y^k)\in (\R^d)^k. $$
Let $\W_{\ell_\alpha}$ be the Wasserstein distance induced by $\rho_{\ell_\alpha}$, i.e.
$$\W_{\ell_\alpha}(\mu,\nu)= \inf_{\pi\in \C(\mu,\nu)}\int_{(\R^{d})^k\times(\R^{d})^k} \rho_{\ell_\alpha}(x,y) \pi(\d x,\d y),\ \ \mu,\nu\in\scr P_{\alpha}((\R^d)^k).$$
Let us recall some results on quantitative propagation of chaos. When $X_0^{i,N}=X_0^i, 1\leq i\leq N$, $b_t(x,\mu)=\int_{\R^d}\tilde{b}_t(x,y)\mu(\d y)$, $\sigma_t(x,\mu)=\int_{\R^d}\tilde{\sigma}_t(x,y)\mu(\d y)$ for some functions $\tilde{b},\tilde{\sigma}$ Lipschitz continuous in spatial variables uniformly in time variable, it is well-known that the quantitative propagation of chaos in strong sense holds, i.e. there exists a constant $c>0$ depending on $T$ such that
$$\E|X_t^{1,N}-X_t^{1}|^2\leq \frac{c}{N}, \ \ t\in[0,T],$$
see for instance \cite{McKean67,SZ}.

 In \cite{BJW,JW,JW1}, the authors introduce the entropy method to derive the quantitative propagation of chaos in relative entropy in the case $\sigma=I_{d\times d}$ and the interaction is bounded or singular, and the result is the entropy-entropy propagation of chaos:
\begin{align}\label{E-E}
\nonumber&\mathrm{Ent}(\L_{(X_t^{1,N},X_t^{2,N},\cdots,X_t^{k,N})}|\L_{(X_t^1,X_t^2,\cdots, X_t^k)})\\
&\leq \frac{ck}{N}+ck \frac{\mathrm{Ent}(\L_{(X_0^{1,N},X_0^{2,N},\cdots,X_0^{N,N})}|\L_{(X_0^1,X_0^2,\cdots, X_0^N)})}{N},\ \ t\in[0,T], 1\leq k\leq N
\end{align}
for some constant $c>0$ depending on $T$, where the relative entropy between $\mu,\nu\in\scr P(E)$ is defined as
$$\mathrm{Ent}(\nu|\mu)=\left\{
  \begin{array}{ll}
    \nu(\log(\frac{\d \nu}{\d \mu})), & \hbox{$\nu$ \text{is absolutely continuous w.r.t.} $\mu$;} \\
    \infty, & \hbox{otherwise.}
  \end{array}
\right.$$
To derive the quantitative rate of $O(k/N)$ in \eqref{E-E}, one need assume that
\begin{align}\label{entro}\sup_{N\ge1 }\mathrm{Ent}(\L_{(X_0^{1,N},X_0^{2,N},\cdots,X_0^{N,N})}|\L_{(X_0^1,X_0^2,\cdots, X_0^N)})<\infty.
\end{align}
We also mention that the authors in \cite{CFGZW} derive the propagation of chaos in relative entropy when the diffusion coefficients are interacting and the Landau equation is covered.
When  $\sigma=I_{d\times d}$ and the interaction is bounded or Lipschitz continuous, \cite{L21} adopts the BBGKY hierarchy to derive the sharp rate $O((k/N)^2)$ of $\mathrm{Ent}(\L_{(X_t^{1,N},X_t^{2,N},\cdots,X_t^{k,N})}|\L_{(X_t^1,X_t^2,\cdots, X_t^k)})$ for $t>0$ if so is at $t=0$.

Observe that \eqref{entro} and the result in \cite{L21} exclude the case $X_0^{i}=x\in\R^d, X_0^{i,N}=x+\frac{1}{N}, 1\leq i\leq N$. Inspired by the log-Harnack inequality \eqref{WHloa} below, \cite{HX23e} studied the entropy-cost type propagation of chaos with Lipschitz continuous interaction and multiplicative noise.
In general,
let $\W_1^\Psi$ be the $L^1$-transportation cost induced by a cost function $\Psi$ (a measurable function $\Psi:E\times E\to[0,\infty)$ with $\Psi(x,x)=0, x\in E$), which is defined by
 $$\W_1^\Psi(\mu,\nu)=\inf_{\pi\in \C(\mu,\nu)} \int_{E\times E} \Psi(x,y) \pi(\d x,\d y),\ \ \mu,\nu\in\scr P(E).$$
The entropy-cost type propagation of chaos is formulated as
\begin{align}\label{ECE}\nonumber&\mathrm{Ent}(\L_{(X_{t}^{1,N},X_{t}^{2,N},\cdots, X_{t}^{k,N})}|\L_{(X_{t}^{1},X_{t}^{2},\cdots,X_{t}^{k})})\\
&\leq k g(N)+kh(t)\frac{\W_1^\Psi(\L_{(X_0^{1,N},X_0^{2,N},\cdots,X_0^{N,N})},\L_{(X_0^{1},X_0^{2},\cdots,X_0^{N})})}{N},\ \ t\in(0,T], 1\leq k\leq N
\end{align}
 for some decreasing function  $g:[0,\infty)\to[0,\infty)$ with $\lim_{N\to \infty}g(N)=0$, some measurable function $h:(0,\infty)\to(0,\infty)$ with $\lim_{t\to 0}h(t)=\infty$ and the $L^1$-transportation cost $\W_1^\Psi$ induced by some cost function $\Psi$.
The function $h(t)$ reflects the regularity of the noise in the interacting system. The entropy-cost propagation of chaos means that the Kac's chaotic property in relative entropy is generated at any positive time even if it does not hold at initial time. Compared wit \eqref{E-E}, the entropy-cost type estimate \eqref{ECE} allows $\L_{(X_0^{1,N},X_0^{2,N},\cdots,X_0^{N,N})}$ not to be absolutely continuous w.r.t. $\L_{(X_0^1,X_0^2,\cdots, X_0^N)}$.

When $b_t(x,\mu)=\int_{\R^d}\tilde{b}_t(x,y)\mu(\d y)$, $\sigma_t(x,\mu)=\sigma(t,x)$ with $\delta^{-1}\leq\sigma\sigma^\ast\leq \delta$ for some $\delta>1$ and $\tilde{b}, \sigma$ are Lipschitz continuous in spatial variable uniformly in time variable, the author has used the coupling by change of measure to derive \eqref{ECE} for $g(N)=\frac{c}{N}, h(t)=\frac{c}{t}$ and $\Psi(x,y)=|x-y|^2, x,y\in(\R^d)^N$ in \cite{HX23e}. In fact, one can easily derive from the Lipschitz continuity of $\tilde{b},\sigma$ that
\begin{align}\label{MYG}\sum_{i=1}^N\E\sup_{t\in[0,T]}|X_t^{i,N}-X_t^{i}|^2\leq c\sum_{i=1}^N\E|X_0^{i,N}-X_0^{i}|^2+c(1+\E|X_0^1|^2),
\end{align}
which together with the Lipschitz continuity of $\tilde{b}$ gives a nice estimate for the crucial quantity
\begin{align}\label{TYM}\E\left|\frac{1}{N}\sum_{m=1}^N\tilde{b}_t(X_t^{i,N}, X_t^{m,N})-\int_{\R^d}\tilde{b}_t(X_t^{i,N}, y)\L_{X_t^1}(\d y)\right|^2.
\end{align}
One can also refer to \cite{CLRW} for the kinetic system with additive noise, where the interaction is also Lipschitz continuous.

However, when $\tilde{b}$ is not Lipschitz continuous, \eqref{MYG} is no longer able to deal with \eqref{TYM}. In this paper, we will develop new technique to derive entropy-cost type propagation of chaos \eqref{ECE} without Lipschitz assumption on the interaction. The main results are Theorem \ref{POC} and Theorem \ref{POC13}. In Theorem \ref{POC}, the interaction is only assumed to be bounded measurable while the cost is an additional independent assumption on the initial values $(X_0^{i,N})_{i\leq i\leq N}$. To remove this assumption in Theorem \ref{POC13}, the interaction is required to be $\alpha$-H\"{o}lder continuous with $\alpha\in(0,\frac{1}{2}]$. In Theorem \ref{POC13}, we also derive some by-products: $\W_{\ell_{\beta}}$-$\W_{\ell_{1}}$ type with $\beta\in(0,1)$ and $TV$-$\W_{\ell_{1}}$ type propagation of chaos, which are of independent interests.


Before moving on, we simply recall Wang's Harnack inequality with power and log-Harnack inequality which are the main tools in the proof of main results. Let $\{P_t\}_{t\geq 0}$ be a Markov semigroup associated to a transition kernel $\{P_t(x,\d y)\}_{t\geq 0,x\in\mathbb{B}}$ on a Banach space $\mathbb{B}$. Wang's Harnack inequality with power is stated as follows: there exists $p>1$ such that
\begin{align}\label{WHa}(P_t f(x))^p\leq P_tf^p(y)\exp\{\Phi_p(t,x,y)\}, \ \ t>0, x,y\in\mathbb{B}, f\in\scr B^+_b(\mathbb{B})
\end{align}
for some measurable function $\Phi_p:(0,\infty)\times\mathbb{B}\times \mathbb{B}\to[0,\infty)$ and $\scr B^+_b(\mathbb{B})$ being the set of all non-negative bounded and measurable functions on $\mathbb{B}$. By \cite[Theorem 1.4.2(1)]{Wbook}, \eqref{WHa} is equivalent to the estimate
\begin{align}\label{Ret}\int_{\mathbb{F}}\left(\frac{\d P_t(x,\d z)}{\d P_t(y,\d z)}\right)^{\frac{p}{p-1}}\d P_t(y,\d z)\leq \exp\left\{\frac{\Phi_p(t,x,y)}{p-1}\right\}, \ \ t>0, x,y\in\mathbb{B}.
\end{align}
The left hand side of \eqref{Ret} is related to the R\'{e}nyi entropy, which is defined as
$$D_p(\mu|\nu)=\frac{1}{p-1}\log \int_{\R^d}(\frac{\d \mu}{\d \nu})^{p-1}\d \mu, \ \ p>0.$$
Note that
$$\mathrm{Ent}(\mu|\nu)=\lim_{p\to 1}D_p(\mu|\nu)=\int_{\R^d} \log \frac{\d
\mu}{\d \nu}\d \mu.$$
So, \eqref{Ret} is also equivalent to R\'{e}nyi entropy-cost inequality:
\begin{align*}(p-1)\log \int_{\mathbb{F}}\left(\frac{\d P_t(x,\d z)}{\d P_t(y,\d z)}\right)^{\frac{p}{p-1}}\d P_t(y,\d z)\leq \Phi_p(t,x,y), \ \ t>0, x,y\in\mathbb{B}.
\end{align*}
As a weaker inequality than \eqref{WHa}, the log-Harnack inequality is formulated as
\begin{align}\label{WHloa}P_t \log f(x)\leq \log P_tf(y)+\tilde{\Phi}(t,x,y), \ \ t>0, x,y\in\mathbb{F}, f\in\scr B^+_b(\mathbb{B}), f\geq 1
\end{align}
for some measurable function $\tilde{\Phi}:(0,\infty)\times\mathbb{B}\times \mathbb{B}\to[0,\infty)$.
By \cite[Theorem 1.4.2(2)]{Wbook}, \eqref{WHloa} is equivalent to
\begin{align}\label{entco}\mathrm{Ent}( P_t(x,\cdot)|P_t(y,\cdot))\leq \tilde{\Phi}(t,x,y), \ \ t>0, x,y\in\mathbb{B}.
\end{align}
One can refer to the monograph \cite{Wbook} for the systematic  introduction of \eqref{WHa} and \eqref{WHloa}. In  \cite{Wbook}, the author derives many other applications of \eqref{WHa} and \eqref{WHloa} such as the uniqueness of invariant probability measure, strong Feller property and gradient estimate and gives plentiful models in which  \eqref{WHa} and \eqref{WHloa} hold.

The remaining of the paper is organized as follows: In Section 2,  we state the main results on entropy-cost type propagation of chaos. Theorem \ref{POC} is concentrated on the case of bounded measurable interaction while Theorem \ref{POC13} consider the H\"{o}lder continuous interaction case. The proofs of them will be provided in Section 3 and Section 4 respectively.
\section{Main results}

Consider the mean field interacting particle system
\begin{align*}\d X^{i,N}_t=b_t(X_t^{i,N}, \ff{1}{N}\sum_{j=1}^N\dd_{X_t^{j,N}})\d t+\sigma_t(X^{i,N}_t) \d W^i_t,\ \ 1\leq i\leq N,
\end{align*}
and the independent McKean-Vlasov SDEs
\begin{align}\label{Eb1}\d X_t^i= b_t(X_t^i, \L_{X_t^i})\d t+  \sigma_t(X^i_t) \d W^i_t,\ \ t\in [0,T],\ \ 1\leq i\leq N.
\end{align}
Let $P_t^\ast\mu_0$ be the distribution of the solution to \eqref{Eb1} with initial distribution $\mu_0\in\scr P(\R^d)$, $(P_t^{[k],N})^\ast\mu_0^N$
be the distribution of $(X_t^{i,N})_{1\leq i\leq k}$ with initial  exchangeable distribution $\mu_0^N\in\scr P((\R^d)^N)$. Let $\mu_0^{\otimes k}$ denote the $k$ independent product of $\mu_0$.

To derive the entropy-cost type propagation of chaos, we make the following assumptions.
\begin{enumerate}
\item[{\bf(A1)}] There exist constants $K_\sigma>0$ and $\delta\geq 1$ such that
\begin{align*}
\|\sigma_t(x)-\sigma_t(y)\|_{HS}\leq K_\sigma|x-y|,\ \ \delta^{-1}\leq\sigma\sigma^\ast\leq \delta, \ \ t\in[0,T], x,y\in\R^d.
\end{align*}
\item[{\bf(A2)}] $b_t(x,\mu)=\bar{b}_t(x)+\int_{\R^d}\tilde{b}_t(x,y)\mu(\d y)$. $\tilde{b}$ is bounded and there exists $K_b>0$ such that
    $$|\bar{b}_t(x)-\bar{b}_t(y)|\leq K_b|x-y|,\ \ |\bar{b}_t(0)|\leq K_b,\ \ t\in[0,T], x,y\in\R^d.$$
\end{enumerate}

\begin{thm}\label{POC} Assume {\bf(A1)}-{\bf(A2)}, $(X_0^{i,N})_{1\leq i\leq N}$ are i.i.d. and $\L_{X_0^{1,N}},\L_{X_0^{1}}\in \scr P_2(\R^d)$. Then there exists a constant $C>0$ depending on $K_b,K_\sigma, d,\delta,T,\|\tilde{b}\|_\infty$ such that for any $\eta\in(0,1)$, $t\in(0,T]$ and $1\leq k\leq N$,
\begin{align}\label{VEN}\|(P_t^{[k],N})^\ast\mu^N_0-(P_t^\ast\mu_0)^{\otimes k}\|_{var}^2
&\leq 2\mathrm{Ent}((P_t^{[k],N})^\ast\mu^N_0|(P_t^\ast\mu_0)^{\otimes k})\\
\nonumber&\leq\frac{Ck}{N}+Ck\left(\frac{t^{1-\eta}}{1-\eta}+\frac{1}{t} \right) \W_1^{\Psi_\eta}(\L_{X_0^{1,N}},\L_{X_0^{1}}),
\end{align}
where $\Psi_\eta(x,y)=|x-y|^{2\eta}+|x-y|^2, x,y\in\R^d$.
\end{thm}
\begin{rem} Compared with \eqref{E-E}, \eqref{VEN} only involves in $\ \W_1^{\Psi_\eta}(\L_{X_0^{1,N}},\L_{X_0^{1}})$ while the price is that $(X_0^{i,N})_{1\leq i\leq N}$ are required to be i.i.d.. Moreover, due to technique reasons, $\eta<1$ in \eqref{VEN} is  crucial and the right hand side will blow up if $\eta=1$. In \cite{HX23e}, when the interaction is Lipschitz continuous, $\W_1^{\Psi_\eta}$ can be replaced by $\W_2^2$ and $(X_0^{i,N})_{1\leq i\leq N}$ are only assumed to be exchangeable.
\end{rem}
To get rid of the assumption that $(X_0^{i,N})_{1\leq i\leq N}$ are i.i.d., the interaction is required to be $\alpha$-H\"{o}lder continuous with small $\alpha\in(0,\frac{1}{2}]$. More precisely, we assume
\begin{enumerate}
    \item[{\bf(A2')}] Besides {\bf(A2)}, there exists a constant $\alpha\in(0,\frac{1}{2}]$ such that
    $$|\tilde{b}_t(x,y)-\tilde{b}_t(\tilde{x},\tilde{y})|\leq K_b(|x-\tilde{x}|^\alpha+|y-\tilde{y}|^\alpha),\ \ t\in[0,T], x,\tilde{x},y,\tilde{y}\in\R^d.$$
    \end{enumerate}

\begin{thm}\label{POC13} Assume {\bf(A1)}, {\bf(A2')}, $\mu_0\in \scr P_{2}(\R^d)$ and $\mu_0^N\in\scr P_1((\R^d)^N)$ be exchangeable. Then the following assertions hold.
\begin{enumerate}
\item[(i)]For any $\beta\in(0,1]$, there exists a constant $c>0$ depending on $K_\sigma, K_b, d,\delta,T,\beta,\alpha,\|\tilde{b}\|_\infty$ such that
\begin{align}\label{nny}
\nonumber& \W_{\ell_\beta}((P_t^{[k],N})^\ast\mu^N_0,(P_t^\ast\mu_0)^{\otimes k})\\
&\leq \frac{ck}{N}(t\wedge 1)^{-\frac{1-\beta}{2}}\W_{\ell_1}(\mu^N_0,\mu_0^{\otimes N})+ \frac{ck}{\sqrt{N}},\ \ t\in(0,T], 1\leq k\leq N.
\end{align}
Moreover, it holds
\begin{align}\label{nny13}
& \|(P_t^{[k],N})^\ast\mu^N_0-(P_t^\ast\mu_0)^{\otimes k}\|_{var}\leq \frac{ck}{N}(t\wedge 1)^{-\frac{1}{2}}\W_{\ell_1}(\mu^N_0,\mu_0^{\otimes N})+ \frac{ck}{\sqrt{N}}
\end{align}
for some constant $c>0$ depending on $K_\sigma, K_b, d,\delta,T,\alpha,\|\tilde{b}\|_\infty$
\item[(ii)]If in addition, $\mu_0^N\in\scr P_2((\R^d)^N)$. Then there exists a constant $C>0$ depending on $K_b,K_\sigma, d,\delta,T,\alpha,\|\tilde{b}\|_\infty$ such that
\begin{align}\label{VENtk}\nonumber&\mathrm{Ent}((P_t^{[k],N})^\ast\mu^N_0|(P_t^\ast\mu_0)^{\otimes k})\\
&\leq\frac{Ck}{\sqrt{N}}+C\frac{k}{N}\left(t^{\frac{1+2\alpha}{2}}+\frac{1}{t} \right) \W_1^{\Psi}(\mu_0^N,\mu_0^{\otimes N}), \ \ t\in(0,T], 1\leq k\leq N,
\end{align}
where $\Psi(x,y)=\sum_{i=1}^N(|x^i-y^i|+|x^i-y^i|^2), x=(x^1,x^2,\cdots,x^N), y=(y^1,y^2,\cdots,y^N)\in(\R^d)^N$.
\end{enumerate}
\end{thm}
\section{Proof of Theorem \ref{POC}}
Before proving Theorem \ref{POC}, we first present two lemmas which will be used in the sequel. The following large deviation type estimate, which can be viewed as the law of large number in exponential scale, is very useful in the proof of quantitative propagation of chaos in relative entropy in the case with bounded interaction, see \cite[Theorem 3]{JW}. We should remark that although \cite[Theorem 3]{JW} only considers the state space $\Pi^d$, the $d$-dimensional torus,  it is a universal result and also available in general Banach spaces since its proof only requires the i.i.d. property of the mentioned random variables.
 \begin{lem}\label{EUP} Let $\mathbb{B}$ be a Banach space. Assume that $\{\xi_i\}_{i\geq 1}$ are i.i.d. $\mathbb{B}$-valued random variables. Then for any bounded measurable function $\phi:\mathbb{B}\to\R$ with $\|\phi\|_\infty \leq (8\e)^{-1}$, we have
$$\E\exp\left\{ N\left|\frac{1}{N}\sum_{i=1}^N\phi(\xi_1,\xi_i)-\int_{\mathbb{B}}\phi(\xi_1,y)\L_{\xi_1}(\d y)\right|^2\right\}\leq 3.$$
\end{lem}
\begin{proof}Let $\psi(z,x)=\phi(z,x)-\int_{\mathbb{B}}\phi(z,z')\L_{\xi_1}(\d z')$. Then we get $\int_{\mathbb{B}}\psi(z,x)\L_{\xi_1}(\d x)=0, z\in\mathbb{B}$ and $\|\psi\|_\infty\leq 2\|\phi\|_\infty\leq (4\e)^{-1}$. Applying \cite[Theorem 3]{JW}, we have
\begin{align*}
&\E\exp\left\{ N\left|\frac{1}{N}\sum_{i=1}^N\phi(\xi_1,\xi_i)-\int_{\mathbb{B}}\phi(\xi_1,y)\L_{\xi_1}(\d y)\right|^2\right\}\\
&\leq 2\left(1+\frac{10(\e\|\psi\|_\infty)^4}{(1-(\e\|\psi\|_\infty)^4)^3}+\frac{4(\e\|\psi\|_\infty)^4} {1-4(\e\|\psi\|_\infty)^4}\right)\\
&\leq 2\left(1+\frac{10\times4^{-4}}{(1-4^{-4})^3}+\frac{4\times4^{-4}} {1-4\times4^{-4}}\right)\leq 3.
\end{align*}
So, we complete the proof.
\end{proof}

The next lemma characterizes the relationship of $L^1$-transportation costs between a single particle and $N$ particles, which is of independent interest. One can also refer to \cite[Proposition 2.6]{HS} for some similar results.
\begin{lem}\label{transport} Let $(E,\rho)$ be a Polish space and $m$ be a positive integer. For any $1\leq i\leq m$, let $\Psi^i:E\times E\to[0,\infty)$ be measurable and define
$$\tilde{\Psi}(x,y)=\sum_{i=1}^m\Psi^i(x^i,y^i),\ \ x=(x^1,x^2,\cdots,x^m),y=(y^1,y^2,\cdots,y^m)\in E^m.$$
Then it holds
$$\W_1^{\tilde{\Psi}}\left(\prod_{i=1}^m \mu^i, \prod_{i=1}^m \nu^i\right)=\sum_{i=1}^m\W_1^{\Psi^i}(\mu^i, \nu^i),\ \ \mu^i,\nu^i\in\scr P(E), 1\leq i\leq m.$$
\end{lem}
\begin{proof} For any $1\leq i\leq m$, let $P_i(x)=x^i,x=(x^1,x^2,\cdots,x^m)\in E^m$. For any $\pi\in\C(\prod_{i=1}^m \mu^i, \prod_{i=1}^m \nu^i)$, it is clear that
$$\pi\circ(P_i)^{-1}\in \C(\mu^i, \nu^i),\ \ 1\leq i\leq m.$$
This implies that
\begin{align*}\sum_{i=1}^m\W_1^{\Psi^i}(\mu^i, \nu^i)&\leq \inf_{\pi\in\C(\prod_{i=1}^m \mu^i, \prod_{i=1}^m \nu^i)}\sum_{i=1}^m\int_{E\times E}\Psi^i(x^i,y^i)(\pi\circ(P_i)^{-1})(\d x^i,\d y^i)\\
&=\W_1^{\tilde{\Psi}}\left(\prod_{i=1}^m \mu^i, \prod_{i=1}^m \nu^i\right).
\end{align*}
On the other hand, for any $\pi_i\in \C(\mu^i, \nu^i),\ \ 1\leq i\leq m$, it is clear that
$$\prod_{i=1}^m\pi_i\in\C\left(\prod_{i=1}^m \mu^i, \prod_{i=1}^m \nu^i\right).$$
Then we have
\begin{align*}\W_1^{\tilde{\Psi}}\left(\prod_{i=1}^m \mu^i, \prod_{i=1}^m \nu^i\right)&\leq\inf_{\pi_i\in \C(\mu^i, \nu^i),1\leq i\leq m}\sum_{i=1}^m\int_{E\times E}\Psi^i(x,y)\d \pi_i(\d x,\d y)\\
&=\sum_{i=1}^m\inf_{\pi_i\in \C(\mu^i, \nu^i)}\int_{E\times E}\Psi^i(x,y)\d \pi_i(\d x,\d y)\\
&=\sum_{i=1}^m\W_1^{\Psi^i}(\mu^i, \nu^i).
\end{align*}
Therefore, the proof is completed.
\end{proof}
Now, we are in the position to prove Theorem \ref{POC}.
\begin{proof}[Proof of Theorem \ref{POC}]
Let $\mu_t^i=\L_{X_t^i},i\geq 1$. Since for any $i\geq 1$, \eqref{Eb1} is well-posed in $\scr P_2(\R^d)$ due to \cite[Theorem 1.1]{23R} for $V=(1+|\cdot|^2)$, $b^{(0)}=0$, $b^{(1)}_t(x,\mu)=\bar{b}(x)+\int_{\R^d}\tilde{b}_t(x,y)\mu(\d y)$, $\mu_t^i$ does not depend on $i$ and we write $\mu_t=\mu_t^i$.
Define
\begin{align*}\P^{0}:= \P(\ \cdot\ |\F_0),\ \ \E^{0}:= \E(\ \cdot\ | \F_0).
\end{align*}
When different probability spaces are involved in, we will use $\L_{\xi|\P}$ to denote the distribution of $\xi$ under $\P$.
To complete the proof, we introduce
\begin{align}\label{GPSer}\d \bar{X}_t^{i}&= \bar{b}_t(\bar{X}_t^{i})\d t+\int_{\R^d}\tilde{b}_t(\bar{X}_t^{i}, y)\mu_t(\d y)\d t+  \sigma_t(\bar{X}^{i}_t) \d W^i_t,\ \ \bar{X}_0^{i}=X_0^{i,N}, 1\leq i\leq N.
\end{align}
We divide the proof into three steps:

{\bf(Step (1))} Estimate $\log\int_{(\R^{d})^N}\left(\frac{\d \L_{(\bar{X}_t^{i})_{1\leq i\leq N}|\P^0}}{\d \L_{(X_t^i)_{1\leq i\leq N}|\P^0}}\right)^{\frac{p}{p-1}}\d \L_{(X_t^i)_{1\leq i\leq N}|\P^0}$ for large enough $p>1$.

{\bf (Step (2))} Estimate $\mathrm{Ent}(\L_{(X_t^{i,N})_{1\leq i\leq N}|\P^0}|\L_{(\bar{X}_t^i)_{1\leq i\leq N}|\P^0})$.

{\bf (Step (3))} Estimate $\mathrm{Ent}(\L_{(X_t^{i,N})_{1\leq i\leq k}}|\L_{(X_t^i)_{1\leq i\leq k}})$ for any $1\leq k\leq N$.

We will finish the above three steps one by one.

(i) Let
\begin{align}\label{bmu}\tilde{b}^\mu_t(x)=\int_{\R^d}\tilde{b}_t(x,y)\mu_t(\d y),\ \ t\in[0,T].x\in\R^d.
\end{align}
Observe
\begin{align}\label{class}
\d X_t^{i}=\bar{b}_t(X_t^{i})\d t+\tilde{b}_t^{\mu}(X_t^{i})\d t+\sigma_t(X_t^{i})\d W_t^i,\ \ 1\leq i\leq N, X_0^i=X_0^i,
\end{align}
and
\begin{align}\label{class12}
\d \bar{X}_t^{i}=\bar{b}_t(\bar{X}_t^{i})\d t+\tilde{b}_t^{\mu}(\bar{X}_t^{i})\d t+\sigma_t(\bar{X}_t^{i})\d W_t^i,\ \ 1\leq i\leq N, \bar{X}_0^i=X_0^{i,N}.
\end{align}
By {\bf(A1)}-{\bf(A2)} and \cite[Theorem 2.2]{23R} for $b^{(0)}=\tilde{b}^\mu$, $b^{(1)}=\bar{b}$, we can find  constants $p>1, c(p)>0$ independent of $N$ and $\mu_t$ such that Wang's Harnack inequality with power $p$ holds, i.e.
\begin{align}\label{Har0}\nonumber\left(\E^{0} [f(\bar{X}_t^{i})]\right)^p
&\leq \E^{0} [f(X_t^{i})^p]\\
&\times\exp\left\{\frac{c(p)|X_0^{i,N}-X_0^{i}|^2}{t}\right\}, \ \ f\in \scr B^+_b(\R^d),t\in(0,T], 1\leq i\leq N.
\end{align}
Since both $(\bar{X}_t^{i})_{1\leq i\leq N}$ and $(X_t^{i})_{1\leq i\leq N}$ are independent under $\P^0$, we derive from \eqref{Har0} that
\begin{align*}\left(\E^{0} [F(\bar{X}_t^{1},\bar{X}_t^{2},\cdots,\bar{X}_t^{N})]\right)^p
&\leq \E^{0} [F(X_t^{1},X_t^{2},\cdots,X_t^{N})^p]\\
&\times\exp\left\{\frac{c(p)\sum_{i=1}^N|X_0^{i,N}-X_0^{i}|^2}{t}\right\}, \ \ F\in \scr B^+_b((\R^d)^N),t\in(0,T].
\end{align*}
This together with the equivalence between \eqref{WHa} and \eqref{Ret} implies for any $t\in(0,T]$,
\begin{align}\label{kyg}&(p-1)\log\int_{(\R^{d})^N}\left(\frac{\d \L_{(\bar{X}_t^{i})_{1\leq i\leq N}|\P^0}}{\d \L_{(X_t^i)_{1\leq i\leq N}|\P^0}}\right)^{\frac{p}{p-1}}\d \L_{(X_t^i)_{1\leq i\leq N}|\P^0}\leq \frac{c(p)\sum_{i=1}^N|X_0^{i,N}-X_0^{i}|^2}{t}.
\end{align}

(ii) Rewrite \eqref{GPSer} as
\begin{align}\label{MXY}\nonumber  \d \bar{X}_t^{i}&=  \bar{b}_t(\bar{X}_t^{i})\d t+\frac{1}{N}\sum_{m=1}^N\tilde{b}_t(\bar{X}_t^{i}, \bar{X}_t^{m})\d t+  \sigma_t(\bar{X}^{i}_t) \d \hat{W}^i_t,\ \ \bar{X}_0^{i}=X_0^{i,N}, 1\leq i\leq N,
\end{align}
where
\begin{equation*}\begin{split}
&\hat{W}_t^i :=   W_t^i-\int_0^t \gamma_s^i\d s,\\ &\gamma_t^i:=[\sigma_t^\ast(\sigma_t\sigma_t^\ast)^{-1}](\bar{X}^{i}_t) \left(\frac{1}{N}\sum_{m=1}^N\tilde{b}_t(\bar{X}_t^{i}, \bar{X}_t^{m})-\int_{\R^d}\tilde{b}_t(\bar{X}_t^{i}, y)\mu_t(\d y)\right).
\end{split}\end{equation*}
Let $t_0\in(0,T]$ and
\begin{align*}
& R_t:= \exp\left\{\int_0^{t}\sum_{i=1}^N\<\gamma_r^i, \d W^i _r\> -\ff 1 2 \int_0^{t} \sum_{i=1}^N|\gamma_r^i|^2\d r\right\}, \ \ t\in [0,t_0].
\end{align*}
Since $\tilde{b}$ is bounded, Girsanov's theorem implies that
$(\hat{W}_t^1,\hat{W}_t^2,\cdots, \hat{W}_t^N)_{t\in[0,t_0]}$ is an $(N\times n)$-dimensional Brownian motion under the weighted conditional probability $\d\Q^{0}_{t_0}=R_{t_0}\d\P^{0}$.
So, we have $\L_{(\bar{X}^i_t)_{1\leq i\leq N}|\Q^{0}_{t_0}}=\L_{(X_t^{i,N})_{1\leq i\leq N}|\P^{0}},\ \ t\in[0,t_0]$, which together with Young's inequality and $\delta^{-1}\leq \sigma\sigma^\ast\leq\delta$ implies that for any $0< F\in \B_b((\R^d)^N)$, it holds
\beg{align}\label{DDT} \nonumber&\E^{0} \log F(X_{t_0}^{1,N},X_{t_0}^{2,N},\cdots, X_{t_0}^{N,N})\\
\nonumber&\le \log \E^{0} [ F(\bar{X}_{t_0}^{1},\bar{X}_{t_0}^{2},\cdots,\bar{X}_{t_0}^{N})]+\E^{0} [R_{t_0}\log R_{t_0}] \\
\nonumber &\leq\log \E^{0} [ F(\bar{X}_{t_0}^{1},\bar{X}_{t_0}^{2},\cdots,\bar{X}_{t_0}^{N})]\\
 &+ \frac{\delta}{2} \sum_{i=1}^N\E_{\Q^{0}_{t_0}}\int_0^{t_0}\left|\frac{1}{N}\sum_{m=1}^N\tilde{b}_t(\bar{X}_t^{i}, \bar{X}_t^{m})-\int_{\R^d}\tilde{b}_t(\bar{X}_t^{i}, y)\mu_t(\d y)\right|^2\d t\\
\nonumber&=\log \E^{0} [ F(\bar{X}_{t_0}^{1},\bar{X}_{t_0}^{2},\cdots,\bar{X}_{t_0}^{N})]\\
\nonumber &+ \frac{\delta}{2} \sum_{i=1}^N\E^0\int_0^{t_0}\left|\frac{1}{N}\sum_{m=1}^N\tilde{b}_t(X_t^{i,N}, X_t^{m,N})-\int_{\R^d}\tilde{b}_t(X_t^{i,N}, y)\mu_t(\d y)\right|^2\d t\\
\nonumber&=:\log \E^{0} [ F(\bar{X}_{t_0}^{1},\bar{X}_{t_0}^{2},\cdots,\bar{X}_{t_0}^{N})]+I_1.
\end{align}
Let $(e_i)_{1\leq i\leq d}$ be a standard orthogonal basis in $\R^d$ and denote $\nu_t=\L_{\bar{X}_t^i|\P}$. Then it holds
\beg{align*} & I_1\leq \sum_{k=1}^d\delta\sum_{i=1}^N\E^0\int_0^{t_0}\left|\frac{1}{N}\sum_{m=1}^N\<\tilde{b}_t(X_t^{i,N}, X_t^{m,N}),e_k\>-\int_{\R^d}\<\tilde{b}_t(X_t^{i,N}, y),e_k\>\nu_t(\d y)\right|^2\d t\\
&+\delta \sum_{i=1}^N\E^0\int_0^{t_0}\left|\int_{\R^d}\tilde{b}_t(X_t^{i,N}, y)\nu_t(\d y)-\int_{\R^d}\tilde{b}_t(X_t^{i,N}, y)\mu_t(\d y)\right|^2\d t\\
&=:\sum_{k=1}^d\sum_{i=1}^N\frac{\delta}{\varepsilon N }I_{11,\vv}(i,k)+\sum_{i=1}^NI_{12}(i)
\end{align*}
for any $\varepsilon>0$.
Applying Young's inequality, we have
\begin{align*}
I_{11,\vv}(i,k)
& \leq \int_0^{t_0}\mathrm{Ent}(\L_{(X_t^{i,N})_{1\leq i\leq N}|\P^0}|\L_{(\bar{X}_t^i)_{1\leq i\leq N}|\P^0})\d t\\
&+\int_0^{t_0}\log\E^0\exp\left\{\varepsilon N\left|\frac{1}{N}\sum_{m=1}^N\<\tilde{b}_t(\bar{X}_t^{i}, \bar{X}_t^{m}),e_k\>-\int_{\R^d}\<\tilde{b}_t(\bar{X}_t^{i}, y),e_k\>\nu_t(\d y)\right|^2\right\}\d t.
\end{align*}
For simplicity, we denote
\begin{align}\label{gasim}
\Gamma_t(i,k,\vv,N)=\E^0\exp\left\{\varepsilon N\left|\frac{1}{N}\sum_{m=1}^N\<\tilde{b}_t(\bar{X}_t^{i}, \bar{X}_t^{m}),e_k\>-\int_{\R^d}\<\tilde{b}_t(\bar{X}_t^{i}, y),e_k\>\nu_t(\d y)\right|^2\right\}.
\end{align}
Next, we estimate $I_{12}(i)$. Recall that $\tilde{b}^\mu$ is defined in \eqref{bmu}. Consider the distribution independent SDE
\begin{align*}
\d Y_t^y=\bar{b}_t(Y_t^y)\d t+\tilde{b}_t^{\mu}(Y_t^y)\d t+\sigma_t(Y_t^y)\d W_t,\ \ Y_0^y=y\in\R^d.
\end{align*}
Set
$$P_t^\mu f(y)=\E f(Y_t^y),\ \ f\in\scr B_b(\R^d), t\in[0,T], y\in\R^d.$$
Applying \cite[Theorem 4.1]{ZY} for $b_0=\tilde{b}^\mu$, $b_1=\bar{b}$ and Pinsker's inequality, there exists a constant $c>0$ independent of $\mu_t$ such that
\begin{align*}|P_t^\mu f(y)-P_t^\mu f(\bar{y})|\leq c\frac{\|f\|_\infty}{\sqrt{t}}|y-\bar{y}|,\ \ f\in\scr B_b(\R^d),t\in(0,T],y,\bar{y}\in\R^d.
\end{align*}
This implies that
\begin{align*}
\nonumber&|P_t^\mu f(y)-P_t^\mu f(\bar{y})|\\
&\leq (2\|f\|_\infty)^{1-\eta}|P_t^\mu f(y)-P_t^\mu f(\bar{y})|^\eta\\
\nonumber&\leq 2^{1-\eta}c^\eta\|f\|_\infty t^{-\frac{\eta}{2}}|y-\bar{y}|^\eta,\ \ f\in\scr B_b(\R^d), t\in(0,T], y,\bar{y}\in\R^d, \eta\in(0,1).
\end{align*}
Combining this with \eqref{class}-\eqref{class12} and the fact
$$\mu_t(f) =\int_{\R^d}P_t^\mu f(y)\L_{X_0^i}(\d y),\ \ \nu_t(f) =\int_{\R^d}P_t^\mu f(y)\L_{X_0^{i,N}}(\d y), \ \ f\in\scr B_b(\R^d), 1\leq i\leq N,$$
we derive
\begin{align*}\|\mu_t-\nu_t\|_{var}\leq 2^{1-\eta}c^\eta t^{-\frac{\eta}{2}}\W_{\eta}(\L_{X_0^{i,N}},\L_{X_0^i}), \ \ t\in(0,T], 1\leq i\leq N, \eta\in(0,1).
\end{align*}
This together with the fact that $\tilde{b}$ is bounded implies
\begin{align}\label{CES} I_{12}(i)
\nonumber&\leq \delta \int_0^{t_0}\|\tilde{b}\|_\infty^2\|\nu_t-\mu_t\|_{var}^2\d t\\
&\leq \delta \|\tilde{b}\|_\infty^24^{1-\eta}c^{2\eta}\W_{\eta}(\L_{X_0^{i,N}},\L_{X_0^i})^2 \frac{t_0^{1-\eta}}{1-\eta},\ \ \eta\in(0,1).
\end{align}
Combining  \eqref{DDT}-\eqref{gasim}, \eqref{CES} with the equivalence between \eqref{WHloa} and \eqref{entco}, we derive
\begin{align*}&\mathrm{Ent}(\L_{(X_{t_0}^{i,N})_{1\leq i\leq N}|\P^0}|\L_{(\bar{X}_{t_0}^i)_{1\leq i\leq N}|\P^0})\\
&\leq \frac{d\delta}{\varepsilon}\int_0^{t_0}\mathrm{Ent}(\L_{(X_t^{i,N})_{1\leq i\leq N}|\P^0}|\L_{(\bar{X}_t^i)_{1\leq i\leq N}|\P^0})\d t\\
&+\sum_{k=1}^d\frac{\delta}{\varepsilon N}\sum_{i=1}^N\int_0^{t_0}\log\Gamma_t(i,k,\vv,N)\d t+\delta N\|\tilde{b}\|_\infty^24^{1-\eta}c^{2\eta}\W_{\eta}(\L_{X_0^{i,N}},\L_{X_0^i})^2 \frac{t_0^{1-\eta}}{1-\eta}.
\end{align*}
Using Gr\"{o}nwall's inequality,
we derive
\begin{align}\label{ENE}\nonumber&\mathrm{Ent}(\L_{(X_t^{i,N})_{1\leq i\leq N}|\P^0}|\L_{(\bar{X}_t^i)_{1\leq i\leq N}|\P^0})\\
&\leq \e^{d\frac{\delta}{\varepsilon} t}\frac{\delta}{\varepsilon} \sum_{k=1}^d\frac{1}{ N}\sum_{i=1}^N\int_0^{t}\log\Gamma_s(i,k,\vv,N)\d s\\
\nonumber&+\e^{d\frac{\delta}{\varepsilon} t}\|\tilde{b}\|_\infty^2 \delta N4^{1-\eta}c^{2\eta}\frac{t^{1-\eta}}{1-\eta} \W_{\eta}(\L_{X_0^{i,N}},\L_{X_0^i})^2,\ \ \eta\in(0,1), t\in[0,T].
\end{align}

(iii) By \cite[Lemma 2.1]{23RW}, we have
\begin{align}\label{triage}
\nonumber&\mathrm{Ent}(\L_{(X_t^{i,N})_{1\leq i\leq N}|\P^0}|\L_{(X_t^i)_{1\leq i\leq N}|\P^0})\\
&\leq p\mathrm{Ent}(\L_{(X_t^{i,N})_{1\leq i\leq N}|\P^0}|\L_{(\bar{X}_t^i)_{1\leq i\leq N}|\P^0})\\
\nonumber&+(p-1)\log\int_{(\R^{d})^N}\left(\frac{\d \L_{(\bar{X}_t^{i})_{1\leq i\leq N}|\P^0}}{\d \L_{(X_t^i)_{1\leq i\leq N}|\P^0}}\right)^{\frac{p}{p-1}}\d \L_{(X_t^i)_{1\leq i\leq N}|\P^0}.
\end{align}
This together with \eqref{kyg}, \eqref{ENE} as well as the equivalence between \eqref{WHloa} and \eqref{entco} implies that for any $\eta\in(0,1)$,
\begin{align}\label{CTK}
\nonumber&\E^{0} \log F(X_{t}^{1,N},X_{t}^{2,N},\cdots, X_{t}^{N,N})\\
\nonumber&\leq  \log \E^{0} [ F(X_{t}^{1},X_{t}^{2},\cdots,X_{t}^{N})]\\
\nonumber&+ p\e^{d\frac{\delta}{\varepsilon} t}\frac{\delta}{\varepsilon} \sum_{k=1}^d\frac{1}{ N}\sum_{i=1}^N\int_0^{t}\log\Gamma_s(i,k,\vv,N)\d s\\
&+p\e^{d\frac{\delta}{\varepsilon} t}\|\tilde{b}\|_\infty^2 \delta N4^{1-\eta}c^{2\eta}\frac{t^{1-\eta}}{1-\eta} \W_{\eta}(\L_{X_0^{i,N}},\L_{X_0^i})^2\\
\nonumber&+\frac{c(p)\sum_{i=1}^N|X_0^{i,N}-X_0^{i}|^2}{t}, \ \ 0< F\in \B_b((\R^d)^N), t\in(0,T].
\end{align}
Taking $\varepsilon= (64\e^2\|\tilde{b}\|_{\infty}^2)^{-1}$ and applying Jensen's inequality and Lemma \ref{EUP} for  $\phi= \<\tilde{b},e_k\>\sqrt{\varepsilon}$, we derive from \eqref{gasim} that
\begin{align}\label{CES01}
\nonumber &\E\sum_{k=1}^d\int_0^{t}\log\Gamma_s(i,k,\vv,N)\d s\\
 &\leq \sum_{k=1}^d\int_0^{t}\log\E\exp\left\{\varepsilon N\left|\frac{1}{N}\sum_{m=1}^N\<\tilde{b}_s(\bar{X}_s^{i}, \bar{X}_s^{m}),e_k\>-\int_{\R^d}\<\tilde{b}_s(\bar{X}_s^{i}, y),e_k\>\nu_s(\d y)\right|^2\right\}\d s\\
\nonumber&\leq dt\log3,
\end{align}
here in the last step, we used the fact that $(\bar{X}_t^{i})_{1\leq i\leq N}$ are i.i.d. under $\P$
since so are $(X_0^{i,N})_{1\leq i\leq N}$. Taking expectation on both sides of \eqref{CTK}, noting $\varepsilon^{-1}=64\e^2\|\tilde{b}\|_{\infty}^2$ and using Jensen's inequality and \eqref{CES01}, we conclude
\begin{align}\label{ENP}
\nonumber&\E \log F(X_{t}^{1,N},X_{t}^{2,N},\cdots, X_{t}^{N,N})\\
&\leq  \log \E [ F(X_{t}^{1},X_{t}^{2},\cdots,X_{t}^{N})]\\
\nonumber&+ p\e^{64\e^2\|\tilde{b}\|_{\infty}^2d\delta t}\|\tilde{b}\|_\infty^2\left(64\e^2d\delta t\log 3+ \delta N4^{1-\eta}c^{2\eta}\frac{t^{1-\eta}}{1-\eta} \W_{\eta}(\L_{X_0^{i,N}},\L_{X_0^i})^2\right)\\
\nonumber&+\frac{c(p)\sum_{i=1}^N\E|X_0^{i,N}-X_0^{i}|^2}{t}, \ \ 0< F\in \B_b((\R^d)^N), \eta\in(0,1).
\end{align}
For any $0< f\in \B_b((\R^d)^k)$, take $$F_f(x_1,x_2,\cdots,x_{\lfloor\frac{N}{k}\rfloor k})=\prod_{i=0}^{\lfloor \frac{N}{k}\rfloor-1}f(x_{ik+1},x_{ik+2},\cdots,x_{ik+k}), \ \ (x_1,x_2,\cdots,x_N)\in (\R^d)^N.$$
Since $(X_{t}^{1},X_{t}^{2},\cdots,X_{t}^{N})$ are i.i.d. under $\P$ and $\lfloor \frac{N}{k}\rfloor^{-1}\leq \frac{2k}{N},1\leq k\leq N$, \eqref{ENP} for $F=F_f$ implies
\beg{align}\label{GTY}\nonumber &\E \log f(X_{t}^{1,N},X_{t}^{2,N},\cdots, X_{t}^{k,N})\\
&\le \log \E [ f(X_{t}^{1},X_{t}^{2},\cdots,X_{t}^{k})]\\
\nonumber&+ \frac{2k}{N}p\e^{64\e^2\|\tilde{b}\|_{\infty}^2d\delta t}\|\tilde{b}\|_\infty^2\left(64\e^2d\delta t\log 3+ \delta N4^{1-\eta}c^{2\eta}\frac{t^{1-\eta}}{1-\eta} \W_{\eta}(\L_{X_0^{i,N}},\L_{X_0^i})^2\right)\\
\nonumber&+2\frac{k}{N}\frac{c(p)\W_{2}(\L_{(X_0^{i,N})_{1\leq i\leq N}},\L_{(X_0^{i})_{1\leq i\leq N}})^2}{t},\ \ 0< f\in \B_b((\R^d)^k),
\end{align}
here the term $\W_2$ is due to the fact that \eqref{ENP} holds if we replace  $((X_0^{i,N})_{1\leq i\leq N}, (X_0^{i})_{1\leq i\leq N})$ on the right hand side by $((\tilde{X}_0^{i,N})_{1\leq i\leq N},(\tilde{X}_0^{i})_{1\leq i\leq N})$ satisfying $\L_{(\tilde{X}_0^{i,N})_{1\leq i\leq N}}=\L_{(X_0^{i,N})_{1\leq i\leq N}}$ and $\L_{(\tilde{X}_0^{i})_{1\leq i\leq N}}=\L_{(X_0^{i})_{1\leq i\leq N}}$.
Finally, since both $(X_0^{i,N})_{1\leq i\leq N}$ and $(X_0^{i})_{1\leq i\leq N}$ are i.i.d., Lemma \ref{transport} for $E=\R^d, m=N, \Psi^i(x,y)=|x-y|^2,x,y\in\R^d$ implies
\begin{align*}\W_{2}(\L_{(X_0^{i,N})_{1\leq i\leq N})},\L_{(X_0^{i})_{1\leq i\leq N})})^2=\sum_{i=1}^N\W_{2}(\L_{X_0^{i,N}},\L_{X_0^{i}})^2=N\W_{2}(\L_{X_0^{1,N}},\L_{X_0^{1}})^2.
\end{align*}
This together with \eqref{GTY}, the equivalence between \eqref{WHloa} and \eqref{entco} and Pinsker's inequality implies \eqref{VEN}.
\end{proof}
\begin{rem}\label{Compa} Combining \eqref{DDT} with \eqref{kyg}, \eqref{triage}, one may derive
\begin{align*}
\nonumber&\mathrm{Ent}(\L_{(X_t^{i,N})_{1\leq i\leq N}}|\L_{(X_t^i)_{1\leq i\leq N}})\\
&\leq \frac{\delta}{2}p \sum_{i=1}^N\int_0^{t}\E\left|\frac{1}{N}\sum_{m=1}^N\tilde{b}_s(X_s^{i,N}, X_s^{m,N})-\int_{\R^d}\tilde{b}_s(X_s^{i,N}, y)\mu_s(\d y)\right|^2\d t\\
\nonumber&+\frac{c(p)\E\sum_{i=1}^N|X_0^{i,N}-X_0^{i}|^2}{t}, \ \ t\in(0,T].
\end{align*}
For the first term on the right hand side, if we apply Young's inequality as in \cite{JW1}, we get
\begin{align*}
&\E\left|\frac{1}{N}\sum_{m=1}^N\tilde{b}_t(X_t^{i,N}, X_t^{m,N})-\int_{\R^d}\tilde{b}_t(X_t^{i,N}, y)\mu_t(\d y)\right|^2\\
&\leq \mathrm{Ent}(\L_{(X_t^{i,N})_{1\leq i\leq N}}|\L_{(X_t^i)_{1\leq i\leq N}})+\log\E\e^{\left|\frac{1}{N}\sum_{m=1}^N\tilde{b}_t(X_t^{i}, X_t^{m})-\int_{\R^d}\tilde{b}_t(X_t^{i}, y)\mu_t(\d y)\right|^2}.
\end{align*}
 Combining this with \eqref{CTK}, one may expect to apply Gronwall's inequality. However, in the present case, this may made $\mathrm{Ent}(\L_{(X_t^{i,N})_{1\leq i\leq N}}|\L_{(X_t^i)_{1\leq i\leq N}})$ blow up due to the singular term $\frac{c(p)\E\sum_{i=1}^N|X_0^{i,N}-X_0^{i}|^2}{t}$ ($\int_0^r\frac{1}{t}\d t=\infty$).
In \cite{JW1}, this does not happen since $\frac{c(p)\E\sum_{i=1}^N|X_0^{i,N}-X_0^{i}|^2}{t}$ is replaced by $\mathrm{Ent}(\mu_0^N|\mu_0^{\otimes N})$. This is the reason why we adopt a different technique to complete the proof of Theorem \ref{POC}.
\end{rem}
\section{Proof of Theorem \ref{POC13}}
\subsection{Proof of Theorem \ref{POC13}(i)}
Before proving Theorem \ref{POC13}, we first present a lemma on law of large number in $L^2$.
 \begin{lem}\label{EUP13} Let $\mathbb{B}$ be a Banach space. Assume that $\{\xi_i\}_{i\geq 1}$ are i.i.d. $\mathbb{B}$-valued random variables. Then for any bounded measurable function $\phi:\mathbb{B}\times\mathbb{B}\to\R$, there exists a constant $c>0$ depending on $\|\phi\|_\infty$ such that
$$\E\left\{\left|\frac{1}{N}\sum_{i=1}^N\phi(\xi_1,\xi_i)-\int_{\mathbb{B}}\phi(\xi_1,y)\L_{\xi_1}(\d y)\right|^2\right\}\leq \frac{c}{N}.$$
\end{lem}
The proof of Lemma \ref{EUP13} is trivial since $\phi$ is bounded so that we omit it.

To complete the proof of Theorem \ref{POC13}(i), for any $s\geq 0$, consider the decoupled SDE
\begin{align}\label{deSDE}\d X_{s,t}^{i,\mu,z}&=\bar{b}_t(X_{s,t}^{i,\mu,z})\d t+\int_{\R^d}\tilde{b}_t(X_{s,t}^{i,\mu,z},y)\mu_t(\d y)\d t+\sigma(X_{s,t}^{i,\mu,z})\d W_t^i, \ \ t\geq s
\end{align}
with $X_{s,s}^{i,\mu,z}=z\in\R^d$.
Let
$$P_{s,t}^{i,\mu} f(z):=\E f(X_{s,t}^{i,\mu,z}), \ \ f\in \scr B_b(\R^{d}),z\in\R^d,i\geq 1,0\leq s\leq t.$$

Let $ C_b^2(\R^d)$ be the set of all continuous functions on $\R^d$ with bounded and continuous up to second order derivatives.
For any $F\in C^1((\R^d)^k)$, $1\leq i\leq k$, $x=(x^1,x^2,\cdots,x^k)\in(\R^d)^k$, let $\nabla_i F(x)$ denote the gradient with respect to $x^i$.
Since \eqref{deSDE} is well-posed so that $P_{s,t}^{i,\mu}$ does not depend on $i$ and we denote $P_{s,t}^{\mu}=P_{s,t}^{i,\mu},\ \ i\geq 1$.
For any $k\geq 1$, $x=(x^1,x^2,\cdots,x^k)\in (\R^{d})^k, F\in \scr B_b((\R^{d})^k)$ and $s\in[0,t]$, define
\begin{align*}(P_{s,t}^\mu)^{\otimes k} F(x):=\E F(X_{s,t}^{1,\mu,x^1},X_{s,t}^{2,\mu,x^2},\cdots,X_{s,t}^{k,\mu,x^k}),\ \ 0\leq s\leq t.
\end{align*}
For simplicity, we write $P_{t}^\mu =P_{0,t}^\mu $.
\begin{proof}[Proof of Theorem \ref{POC13}(i)]
For any $1\leq i\leq N$, $s\in[0,T]$, let
  \begin{align}\label{Bi1}B^i_s(x)=\frac{1}{N}\sum_{m=1}^N\tilde{b}_s(x^i,x^m) -\int_{\R^d}\tilde{b}_s(x^i,y)\mu_s(\d y),\ \ x=(x^1,x^2,\cdots,x^N)\in(\R^d)^N.
  \end{align}
    Then it follows from {\bf(A2')} that
\begin{align}\label{cti}
\sup_{x\neq y}\frac{|\sum_{i=1}^N|B_s^i(x)|-\sum_{i=1}^N|B_s^i(y)||}{\rho_{\ell_\alpha}(x,y)}\leq 3K_b.
\end{align}
Let $\beta\in(0,1]$ and $1\leq k\leq N$.
Let $\pi_k$ be the projection mapping defined by
$$\pi_k(x)=(x^1,x^2,\cdots,x^k),\ \ x=(x^1,x^2,\cdots,x^N)\in(\R^d)^N.$$
For any $F\in C_b^2((\R^d)^k)$ and $t\geq 0$, it holds
\begin{align}\label{DUH0}
\nonumber&\int_{(\R^d)^k}F(x)\{(P_t^{[k],N})^\ast\mu^N_0\}(\d x)-\int_{(\R^d)^k}\{(P_{t}^\mu)^{\otimes k} F\}(x)(\mu^N_0\circ\pi_k^{-1})(\d x)\\
&=\int_0^t\sum_{i=1}^k\int_{(\R^d)^N}\bigg\<B^i_s(x),[\nabla_{i}(P^\mu_{s,t})^{\otimes k}F](\pi_kx)\bigg\>\{(P_s^{[N],N})^\ast\mu^N_0\}(\d x)\d s
    \end{align}
Let $\tilde{b}_s^\mu=\int_{\R^d}\tilde{b}_s(\cdot,y)\mu_s(\d y)$. By {\bf (A1)}-{\bf (A2)}, applying \cite[Theorem 4.1]{ZY} for $b_0=\tilde{b}^\mu$, $b_1=\bar{b}$ and Pinsker's inequality, there exists a constant $c_0>0$ independent of $\mu_t$ such that
\begin{align*}
|\nabla P_{s,t}^\mu f|\leq \left(c_0((t-s)\wedge 1)^{-\frac{1-\beta}{2}}\|f\|_{[\beta]}\right)\wedge \left(c_0((t-s)\wedge 1)^{-\frac{1}{2}}\|f\|_{\infty}\right),
\end{align*}
which implies that for any $1\leq i\leq k$,
\begin{align}\label{cmy}
[\nabla_{i}(P^\mu_{s,t})^{\otimes k}F]\leq \left(c_0((t-s)\wedge 1)^{-\frac{1-\beta}{2}}\|F\|_{[\ell_\beta]}\right)\wedge \left(c_0((t-s)\wedge 1)^{-\frac{1}{2}}\|F\|_{\infty}\right),
\end{align}
here
 $[F]_{[\ell_\beta]}:=\sup_{x\neq y}\frac{|F(x)-F(y)|}{\rho_{\ell_\beta}(x,y)}$.
Combining \eqref{cmy} with \eqref{DUH0} and the fact that $(X_s^{i,N})_{1\leq i\leq N}$ and $(X_s^{i})_{1\leq i\leq N}$ are exchangeable, we conclude that for $[F]_{[\ell_\beta]}\leq 1$,
\begin{align}\label{DUH01}
\nonumber&\int_{(\R^d)^k}F(x)\{(P_t^{[k],N})^\ast\mu^N_0\}(\d x)-\int_{(\R^d)^k}\{(P_{t}^\mu)^{\otimes k} F\}(x)(\mu^N_0\circ\pi_k^{-1})(\d x)\\
\nonumber&\leq c_0\int_0^t((t-s)\wedge 1)^{-\frac{1-\beta}{2}}\sum_{i=1}^k\int_{(\R^d)^N}|B^i_s(x)|\{(P_s^{N})^\ast\mu^N_0\}(\d x)\d s\\
& = c_0\int_0^t((t-s)\wedge 1)^{-\frac{1-\beta}{2}}\frac{k}{N}\sum_{i=1}^N\int_{(\R^d)^N}|B^i_s(x)|\{(P_s^{N})^\ast\mu^N_0-(P_s^\ast\mu_0)^{\otimes N}\}(\d x)\d s\\
&+\nonumber c_0\int_0^t((t-s)\wedge 1)^{-\frac{1-\beta}{2}}k\int_{(\R^d)^N}|B^i_s(x)|\{(P_s^\ast\mu_0)^{\otimes N}\}(\d x)\d s.
\end{align}
By \eqref{cti}, we have 
\begin{align*}
&\int_{(\R^d)^N}\sum_{i=1}^N|B^i_s(x)|\{(P_s^{N})^\ast\mu^N_0-(P_s^\ast\mu_0)^{\otimes N}\}(\d x) \leq 3K_b\W_{\ell_{\alpha}}(P_s^{[N],N})^\ast\mu^N_0,(P_s^\ast\mu_0)^{\otimes N}).
\end{align*}
This combined with \eqref{DUH01} and Lemma \ref{EUP13} implies for $[F]_{[\ell_\beta]}\leq 1$,
\begin{align}\label{DUH03}
\nonumber&\int_{(\R^d)^k}F(x)\{(P_t^{[k],N})^\ast\mu^N_0\}(\d x)-\int_{(\R^d)^k}\{(P_{t}^\mu)^{\otimes k} F\}(x)(\mu^N_0\circ\pi_k^{-1})(\d x)\\
&\leq c_1\int_0^t((t-s)\wedge 1)^{-\frac{1-\beta}{2}}\frac{k}{N}\W_{\ell_{\alpha}}(P_s^{[N],N})^\ast\mu^N_0,(P_s^\ast\mu_0)^{\otimes N})\d s\\
\nonumber&+\frac{c_1k}{\sqrt{N}}\int_0^t((t-s)\wedge 1)^{-\frac{1-\beta}{2}}\d s.
\end{align}
Next, for any $\tilde{\pi}\in \C(\mu^N_0\circ\pi_k^{-1},\mu_0^{\otimes k})$, it is easy to see that 
\begin{align*}&\left|\int_{(\R^d)^k}\{(P_{t}^\mu)^{\otimes k} F\}(x)(\mu^N_0\circ\pi_k^{-1})(\d x)-\int_{(\R^d)^k}\{(P^\mu_{t})^{\otimes k}F\}(x)\mu_0^{\otimes k}(\d x)\right|\\
&\leq \int_{(\R^d)^k\times(\R^d)^k}|\{(P^\mu_{t})^{\otimes k}F\}(x)-\{(P^\mu_{t})^{\otimes k}F\}(y)|\tilde{\pi}(\d x,\d y).
\end{align*}
This together with \eqref{cmy} implies 
\begin{align}\label{tyl}\nonumber&\left|\int_{(\R^d)^k}\{(P_{t}^\mu)^{\otimes k} F\}(x)(\mu^N_0\circ\pi_k^{-1})(\d x)-\int_{(\R^d)^k}\{(P^\mu_{t})^{\otimes k}F\}(x)\mu_0^{\otimes k}(\d x)\right|\\
&\leq \left[\left(c_0(t\wedge 1)^{-\frac{1-\beta}{2}}\|F\|_{[\ell_\beta]}\right)\wedge \left(c_0(t\wedge 1)^{-\frac{1}{2}}\|F\|_{\infty}\right)\right]\W_{\ell_1}(\mu^N_0\circ\pi_k^{-1},\mu_0^{\otimes k}).
\end{align}
Combining \eqref{tyl} with \eqref{DUH03} for $k=N$ and the triangle inequality, we get
\begin{align}\label{cmy13}
\nonumber&\W_{\ell_\beta}((P_t^{[N],N})^\ast\mu^N_0,(P_t^\ast\mu_0)^{\otimes N})\\
&\leq c_0(t\wedge 1)^{-\frac{1-\beta}{2}}\W_{\ell_1}(\mu^N_0,\mu_0^{\otimes N})+ \frac{c_1N}{\sqrt{N}}\int_0^t((t-s)\wedge 1)^{-\frac{1-\beta}{2}}\d s\\
\nonumber&+c_1\int_0^t((t-s)\wedge 1)^{-\frac{1-\beta}{2}}\W_{\ell_{\alpha}}(P_s^{[N],N})^\ast\mu^N_0,(P_s^\ast\mu_0)^{\otimes N})\d s.
\end{align}
For any $\lambda>0$ and $t_0\in(0,T]$, define
$$\Gamma_{t_0}(\lambda)=\sup_{t\in[0,t_0]}\e^{-\lambda t}(t\wedge 1)^{\frac{1-\alpha}{2}}\W_{\ell_\alpha}((P_t^{[N],N})^\ast\mu^N_0,(P_t^\ast\mu_0)^{\otimes N}).$$
It follows from \eqref{cmy13} for $\beta=\alpha$ that
\begin{align}\label{gat}
\nonumber\Gamma_{t_0}(\lambda)&\leq c_0\W_{\ell_1}(\mu^N_0,\mu_0^{\otimes N})+ \frac{2c_1N}{\sqrt{N}}(t_0\vee t_0^{\frac{1+\alpha}{2}})(t_0\wedge 1)^{\frac{1-\alpha}{2}}\\
&+\Gamma_{t_0}(\lambda)c_1\sup_{t\in[0,t_0]}(t\wedge 1)^{\frac{1-\alpha}{2}}\int_0^t\e^{-\lambda (t-s)}((t-s)\wedge 1)^{-\frac{1-\alpha}{2}}(s\wedge 1)^{-\frac{1-\alpha}{2}}\d s.
\end{align}
 Note that there exists a constant $\lambda_0>0$ such that $$c_1\sup_{t\in[0,T]}(t\wedge 1)^{\frac{1-\alpha}{2}}\int_0^t\e^{-\lambda_0 (t-s)}((t-s)\wedge 1)^{-\frac{1-\alpha}{2}}(s\wedge 1)^{-\frac{1-\alpha}{2}}\d s\leq \frac{1}{2}.$$
So, we derive from \eqref{gat} that
$$\Gamma_{t_0}(\lambda_0)\leq 2c_0\W_{\ell_1}(\mu^N_0,\mu_0^{\otimes N})+ \frac{4c_1N}{\sqrt{N}}t_0,$$
which implies 
\begin{align}\label{Walph}
\nonumber&\W_{\ell_\alpha}((P_t^{[N],N})^\ast\mu^N_0,(P_t^\ast\mu_0)^{\otimes N})\\
&\leq \e^{\lambda _0t}2c_0(t\wedge 1)^{-\frac{1-\alpha}{2}}\W_{\ell_1}(\mu^N_0,\mu_0^{\otimes N})+ \e^{\lambda _0t}(t\vee t^{\frac{1+\alpha}{2}})\frac{4c_1N}{\sqrt{N}}, \ \ t\in[0,T].
\end{align}
Substituting \eqref{Walph} into \eqref{cmy13}, we obtain
\begin{align*}
&\W_{\ell_\beta}((P_t^{[N],N})^\ast\mu^N_0,(P_t^\ast\mu_0)^{\otimes N})\leq c(t\wedge 1)^{-\frac{1-\beta}{2}}\W_{\ell_1}(\mu^N_0,\mu_0^{\otimes N})+ \frac{cN}{\sqrt{N}},\ \ t\in[0,T].
\end{align*}
This implies \eqref{nny} due to 
$$\W_{\ell_\beta}((P_t^{[k],N})^\ast\mu^N_0,(P_t^\ast\mu_0)^{\otimes k})\leq \frac{k}{N}\W_{\ell_\beta}((P_t^{[N],N})^\ast\mu^N_0,(P_t^\ast\mu_0)^{\otimes N}), \ \ 1\leq k\leq N.$$

Finally, by \eqref{tyl} and \eqref{DUH03} for $\beta=0$ and using $\W_{\ell_1}(\mu^N_0\circ\pi_k^{-1},\mu_0^{\otimes k})\leq \frac{k}{N}\W_{\ell_1}(\mu^N_0,\mu_0^{\otimes N})$, we arrive at
\begin{align*}
&\|(P_t^{[k],N})^\ast\mu^N_0-(P_t^\ast\mu_0)^{\otimes k}\|_{var}\\
&\leq c_1\int_0^t((t-s)\wedge 1)^{-\frac{1}{2}}\frac{k}{N}\W_{\ell_{\alpha}}(P_s^{[N],N})^\ast\mu^N_0,(P_s^\ast\mu_0)^{\otimes N})\d s\\
&+\frac{c_1k}{\sqrt{N}}\int_0^t((t-s)\wedge 1)^{-\frac{1}{2}}\d s+c_0((t-s)\wedge 1)^{-\frac{1}{2}}\frac{k}{N}\W_{\ell_1}(\mu^N_0,\mu_0^{\otimes N}),
\end{align*}
This combined with \eqref{Walph} implies \eqref{nny13} and the proof is completed.
\end{proof}
\subsection{Proof of Theorem \ref{POC13}(ii)}

\begin{proof}[Proof of Theorem \ref{POC13}(ii)]
Combining \eqref{DDT} with \eqref{kyg}, \eqref{triage} and noting that $$\mathrm{Ent}(\L_{(X_t^{i,N})_{1\leq i\leq N}}|\L_{(X_t^i)_{1\leq i\leq N}})\leq \E\mathrm{Ent}(\L_{(X_t^{i,N})_{1\leq i\leq N}|\P^0}|\L_{(X_t^i)_{1\leq i\leq N}|\P^0}),$$
we derive 
\begin{align}\label{CTK13}
\nonumber&\mathrm{Ent}(\L_{(X_t^{i,N})_{1\leq i\leq N}}|\L_{(X_t^i)_{1\leq i\leq N}})\\
&\leq \frac{\delta}{2}p \sum_{i=1}^N\int_0^{t}\E\left|\frac{1}{N}\sum_{m=1}^N\tilde{b}_s(X_s^{i,N}, X_s^{m,N})-\int_{\R^d}\tilde{b}_s(X_s^{i,N}, y)\mu_s(\d y)\right|^2\d s\\
\nonumber&+\frac{c(p)\E\sum_{i=1}^N|X_0^{i,N}-X_0^{i}|^2}{t}, \ \ t\in(0,T].
\end{align}
Recall \eqref{Bi1}. By {\bf(A2')}, we conclude
\begin{align}\label{bb2}
\nonumber&\sum_{i=1}^N|B^i_s(x)-B^i_s(y)|^2\\
\nonumber&\leq \sum_{i=1}^N\left(2K_b|x^i-y^i|^\alpha+K_b\frac{1}{N}\sum_{m=1}^N|x^m-y^m|^\alpha\right)^2\\
&\leq 8K_b^2\sum_{i=1}^N|x^i-y^i|^{2\alpha}+2K_b^2\sum_{i=1}^N\frac{1}{N}\sum_{m= 1}^N|x^m-y^m|^{2\alpha}\\
\nonumber&=10K_b^2\sum_{i=1}^N|x^i-y^i|^{2\alpha},\ \ x=(x^1,x^2,\cdots,x^N), y=(y^1,y^2,\cdots,y^N)\in(\R^d)^N.
\end{align}
Noting that for any $\pi\in \scr C((P_s^{[N],N})^\ast\mu^N_0,(P_s^\ast\mu_0)^{\otimes N})$, it holds
\begin{align*}
&\E\sum_{i=1}^N\left|\frac{1}{N}\sum_{m=1}^N\tilde{b}_s(X_s^{i,N}, X_s^{m,N})-\int_{\R^d}\tilde{b}_s(X_s^{i,N}, y)\mu_s(\d y)\right|^2\\
&=\int_{(\R^d)^N\times (\R^d)^N}\sum_{i=1}^N|B_s^i(x)|^2\pi(\d x,\d y)\\
& \leq 2\int_{(\R^d)^N\times (\R^d)^N}\sum_{i=1}^N|B_s^i(x)-B^i_s(y)|^2\pi(\d x,\d y)+2\sum_{i=1}^N\int_{(\R^d)^N}|B^i_s(y)|^2(P_s^\ast\mu_0)^{\otimes N}(\d y)\\
&=20K_b^2\int_{(\R^d)^N\times (\R^d)^N}\sum_{i=1}^N|x^i-y^i|^{2\alpha}\pi(\d x,\d y)+c(T),
\end{align*}
where in the last step, we used \eqref{bb2} and $
\int_{(\R^d)^N}|B^i_s(y)|^2(P_s^\ast\mu_0)^{\otimes N}(\d y)\leq \frac{c_0(T)}{N}$ due to Lemma \ref{EUP13}.
This immediately gives 
\begin{align*}
&\E\sum_{i=1}^N\left|\frac{1}{N}\sum_{m=1}^N\tilde{b}_s(X_s^{i,N}, X_s^{m,N})-\int_{\R^d}\tilde{b}_s(X_s^{i,N}, y)\mu_s(\d y)\right|^2\\
&\leq 20K_b^2\W_{\ell_{2\alpha}}((P_s^{[N],N})^\ast\mu^N_0,(P_s^\ast\mu_0)^{\otimes N})+c(T).
\end{align*}
Combining this with \eqref{nny} for $\beta=2\alpha(\alpha\in(0,\frac{1}{2}])$ and \eqref{CTK13}, we get
\beg{align*}&\E \log F(X_{t}^{1,N},X_{t}^{2,N},\cdots, X_{t}^{N,N})\\
&\le \log \E [ F(X_{t}^{1},X_{t}^{2},\cdots,X_{t}^{N})]+ \frac{c(p)\E\sum_{i=1}^N|X_0^{i,N}-X_0^{i}|^2}{t}\\
&+ct^{\frac{1+2\alpha}{2}}\W_{\ell_1}(\mu^N_0,\mu_0^{\otimes N})+ \frac{cN}{\sqrt{N}}+c(T),\ \ 0< F\in \B_b((\R^d)^N).
\end{align*}
Repeating the argument to derive \eqref{GTY} from \eqref{ENP}, the equivalence between \eqref{WHloa} and \eqref{entco}, \eqref{VENtk} holds and the proof is completed.
\end{proof}



\end{document}